\documentclass{amsart}
\usepackage[utf8]{inputenc}

\usepackage{amsmath}
\usepackage{amssymb}
\usepackage{enumitem}
\usepackage{amsthm}
\usepackage{color}
\usepackage{comment}
\usepackage{graphicx}
\usepackage{tikz-cd}
\usepackage{mathabx} 

\numberwithin{equation}{section}

\newtheorem{theorem}{Theorem}[section]
\newtheorem{proposition}[theorem]{Proposition}
\newtheorem{lemma}[theorem]{Lemma}
\newtheorem{corollary}[theorem]{Corollary}
\theoremstyle{definition}
\newtheorem{definition}[theorem]{Definition}
\theoremstyle{remark}
\newtheorem*{remark}{Remark}

\begin{document}

\title[Schoenberg correspondence]{Schoenberg Correspondence for $k$-(Super)Positive Maps on Matrix Algebras}

\author{B.V.Rajarama\ Bhat}
\address{B.V.R.~B.: Indian Statistical Institute, Stat-Math. Unit, R V College Post, Bengaluru 560059, India}
\email{bhat@isibang.ac.in}

\author{Purbayan Chakraborty}
\address{P.C.: Universit\'e de Franche-Comt\'e, CNRS, UMR 6623, LmB, F-25000 Besan\c{c}on, France}
\email{purbayan.chakraborty@univ-fcomte.fr}

\author{Uwe Franz}
\address{U.F.: Universit\'e de Franche-Comt\'e, CNRS, UMR 6623, LmB, F-25000 Besan\c{c}on, France}
\email{uwe.franz@univ-fcomte.fr}
\urladdr{https://lmb.univ-fcomte.fr/Franz-Uwe}

\begin{abstract}
We prove a Schoenberg-type correspondence for non-unital semigroups which generalizes an analogous result for unital semigroup proved by Michael Sch\"urmann \cite{sch}. It characterizes the generators of semigroups of linear maps on $M_n(\mathbb{C})$ which are $k$-positive, $k$-superpositive, or $k$-entanglement breaking. As a corollary we reprove Lindblad, Gorini, Kossakowski, Sudarshan's theorem\cite{lind,gks}.
We present some concrete examples of semigroups of operators and study how their positivity properties can improve with time.
\end{abstract}
\maketitle

\tableofcontents

\section{Introduction}
\label{sec:intro}

We characterise the generators of semigroups of linear maps acting on the matrix algebras $M_n$ which are contained in various cones relevant to quantum information. In their famous papers Lindblad\cite{lind}, Gorini, Kossakowski, Sudarshan\cite{gks} charac\-terised the generators of semigroups of completely positive identity preserving maps. 
The question of semigroups of other quantum maps (e.g. $k$-positive or $k$-superpositive etc.) arose naturally. 
In this paper we address the question of characterising generators of semigroups of quantum maps, and in particular $k$-(super)positive maps.

Let $C$ be a cone in $\mathbb{R}^n$ and denote by $C^{\circ}$ the cone dual to $C$, cf.\ Definition \ref{def dual cone}. Schneider and Vidyasagar\cite{sv} showed that if $\phi$ is an operator on $\mathbb{R}^n$ then the semigroup $(T_t)_{t\ge 0}= \exp{(t\phi)}_{t\ge 0}$ leaves the cone $C$ invariant if and only if $\langle \phi x|y\rangle \ge 0$ for any $x\in C$ and $y \in C^{\circ}$  satisfying $\langle x|y\rangle = 0$.
This already established an important key result for semigroups of operators and their generators in $\mathbb{R}^n$.
This kind of correspondence is known as `Schoenberg correspondence'.
Michael Sch\"urmann extended this result further to unital Banach algebras, coalgebras, and bialgebras.
We can take advantage of these results to establish the correspondence between a general semigroup of $k$-(super)positive maps and its generator. But the problem is that in all of the previous considerations the semigroup starts at the identity. i.e., at time $t=0$, $T_t=\mathrm{Id}$.
As we know, the identity map though completely positive, is not a $k$-super positive map in general.
So to apply Schoenberg type correspondence for a general $k$-super positive semigroup we need to start at some point other than identity which should lie inside the cone of $k$-super positive maps.
The natural choice is an idempotent. We extend the Schoenberg correspondence that Sch\"urmann proved in his paper\cite{sch} to the case where the semigroup starts at $t=0$ with some idempotent contained in the relevant cone $C$.

In Section $2$, we recall the notions of convex cone, dual cone, and some of their properties without going into details. 
Then we shall briefly discuss the cones of $k$-positive maps, completely positive maps, $k$-superpositive maps and $k$-entanglement breaking maps in the space $\mathrm{Lin}(M_n(\mathbb{C}),M_n(\mathbb{C}))$. 
We end Section $2$ by recalling basic facts on non-unital semigroups, and giving an example of a $k$-superpositive idempotent that is not a conditional expectation onto some unital *-subalgebra of $M_n(\mathbb{C})$.
Section $3$ contains the main result of this paper, i.e., the non-unital version of Sch\"urmann's Schoenberg correspondence for semigroups of operators on a unital Banach algebra.
In Section $4$, we apply this result for different cones of interest in quantum information. Furthermore, we characterise the generators of semigroups of $k$-positive maps, and recover the Linblad, Gorini, Kossakowski, Sudarshan's theorem. 
We discuss a characterisation of positive semigroups on $M_2(\mathbb{C})$ in Section $5$.
We end our paper with Section $6$, where study the four parameter family spanned by the depolarizing channel, the transposition, the conditional expectation onto diagonal matrices, and the identity map, and describe the time evolution of semigroups generated by this family.
The Schoenberg correspon\-dence gives a sufficient and necessary condition for a semigroup to lie inside some cone $C$ for all $t\ge0$. But our examples show that the positivity properties of a semigroup can improve over time, i.e., sometimes one can compute some time $t_0>0$ s.t.\ $T_t$ belongs to a subcone $C_0\subseteq C$ for all $t\ge t_0$.

\section{Preliminaries}
\label{sec:prelim}

\subsection{Cones}

We recall some definitions and basic facts on cones.
\begin{definition}\label{def dual cone}
    A subset $C$ of a topological vector space $V$ over $\mathbb{R}$ is called a cone if for any two elements $x,y\in C$ and $\alpha \geq 0$ we have $\alpha x + y \in C$. Furthermore, a cone is called solid if it has nonempty interior, and pointed if $C \cap (-C)=\{0\}$.
\end{definition}

\begin{definition}
For a cone $C\subseteq V$, its dual cone $C^{\circ}$ is defined as
\[
C^{\circ}:= \{z\in V'; \langle z, x\rangle \geq 0, x\in C\},
\]
where $V'$ is the topological dual space of $V$.
\end{definition}
The following results are well known so we mention them without proof. Cf.\ \cite[Lemma 3.2, Corollary 3.3]{aliprantis+tourky}.
\begin{proposition} Let $C$ be a closed convex cone in $\mathbb{R}^n$.
Then the following statements are equivalent:
\begin{itemize}
\item[a.]
$C$ is pointed i.e. $C \cap (-C)=\{0\}$.
\item[b.]
$C^{\circ}- C^{\circ}= \mathbb{R}^n$.
\item[c.]
$C^{\circ}$ has non-empty interior.
\item[d.]
$\rm{span} (C^{\circ})= \mathbb{R}^n$ .
\end{itemize}
\end{proposition}
If we assume $C$ is closed (which is the case for all cones we give as examples) the same results hold if we replace $C$ by $C^{\circ}$ in the above proposition via the Bipolar theorem $(C^{\circ})^{\circ}= C$, see \cite[Theorem 5.5]{sim}.

In the following discussion we will often use Dirac's bra-ket notation to denote rank-one operators. 
To define it in short, we understand $|x\rangle$ to be a vector in some Hilbert space $\mathbb{C}^n$ and $\langle x|$ to be its dual vector.
Then for two vectors $x,y\in \mathbb{C}^n$ we define the rank one operator $|x\rangle \langle y|$ by
\[
|x\rangle \langle y|(z):= \langle y|z \rangle x
\]
for any $z\in \mathbb{C}^n$.
If $\{e_i\}_{i=1}^n$ is the standard basis of $\mathbb{C}^n$ and $\{E_{ij}; 1\le i,j \le n\}$ is the standard matrix units (i.e., $1$ at $ij$-th position and zero everywhere else) then one can easily verify that $E_{ij}=|e_i\rangle\langle e_j|$.

\subsection{The algebra $\mathrm{Lin}(M_n,M_n)$}

We will consider the *-algebra $M_n$ as a Hilbert space with the Hilbert-Schmidt inner product
\[
\langle A,B\rangle = \mathrm{Tr}(A^*B), \qquad A,B\in M_n,
\]
and we will also consider the Hilbert-Schmidt inner product
\[
\langle R,S\rangle = \mathrm{Tr}(R^*S), \qquad R,S\in\mathrm{Lin}(M_n,M_n),
\]
on $\mathrm{Lin}(M_n,M_n)$.

In addition to the involution $*$ on $\mathrm{Lin}(M_n,M_n)$ defined w.r.t.\ the inner product on $M_n$, where $S\in \mathrm{Lin}(M_n,M_n)$ gets mapped to the unique $S^*$ satisfying the condition
\[
\big\langle S^*(A),B\big\rangle = \big\langle A, S(B)\big\rangle, \quad \forall A,B\in M_n,
\]
we will also consider the \emph{multiplicative} involution $\#: \mathrm{Lin}(M_n,M_n)\to\mathrm{Lin}(M_n,M_n)$ defined by
\[
S^\#(A) = S(A^*)^*
\]
for $S\in\mathrm{Lin}(M_n,M_n)$ and $A\in M_n$.
It is multiplicative in the sense that $(S_1\circ S_2)^{\#}= S_1^{\#} \circ S_2^{\#}$ for $S_1,S_2\in \mathrm{Lin}(M_n,M_n)$.

For a pair of matrices $A,B\in M_n$ we define a linear map $T_{A,B}:M_n\to M_n$ by
\[
T_{A,B}(X) = AXB, \qquad \text{ for } X\in M_n.
\]
Denote by $M_n^\mathrm{op}$ the matrix algebra $M_n$ equipped with the opposite multiplication, $A\cdot_\mathrm{op} B = BA$ for $A,B\in M_n$.

\begin{proposition}\label{isomorphism}
The map $M_n\times M_n\ni (A,B)\mapsto T_{A,B}\in \mathrm{Lin}(M_n,M_n)$ extends to a unique isomorphism of *-algebras $T:M_n\otimes M_n^\mathrm{op}\ni A\otimes B \to T_{A\otimes B}\in \mathrm{Lin}(M_n,M_n)$, which is also an isomorphism of Hilbert spaces.

If we define an involution $\dag:M_n\otimes M_n^\mathrm{op}\to M_n\otimes M_n^\mathrm{op}$ on simple tensors $A\otimes B\in M_n\otimes M_n^\mathrm{op}$ by $(A\otimes B)^\dag = B^*\otimes A^*$ and extend conjugate-linearly, then we have
\[
T_{M^\dag} = (T_M)^\#, \qquad M\in M_n\otimes M_n^\mathrm{op}.
\]
\end{proposition}

Denote by $\mathrm{Lin}(M_n,M_n)^{\mathrm{her}}$ the space of hermitianity preserving linear maps, i.e., the space of the linear maps $T:M_n\to M_n$ s.t.\ $T(X^*)=T(X)^*$ for $X\in M_n$, and by $(M_n\otimes M_n^\mathrm{op})^{\mathrm{sa}}$ the space of tensors that are self-adjoint w.r.t.\ $\dag$. Then ismomorphism from the previous proposition also restricts to an isomorphism between $\mathrm{Lin}(M_n,M_n)^{\mathrm{her}}$ and $(M_n\otimes M_n^\mathrm{op})^{\mathrm{sa}}$.

\subsection{The cones of $k$-positive, $k$-superpositive and $k$-entanglement breaking maps on $M_n$}

For linear maps from $M_n$ to $M_n$ we can define the notions of $k$-positivity and $k$-superpositivity, see \cite{ssz} and the references therein.

In $M_n$ we consider the cone of positive matrices
\[
M_n^+=\{A\in M_n:\exists B\in M_n\text{ s.t. } A=B^*B\},
\]
and a linear map $S:M_n\to M_n$ is called \emph{positive}, if it maps positive matrices to positive matrices. We will denote the cone of positive maps by $\mathcal{PM}$.

\begin{definition}
A linear map $S\in\mathrm{Lin}(M_n,M_n)$ is called \emph{$k$-positive}, if $\mathrm{Id}_{M_k}\otimes S : M_k\otimes M_n \to M_k\otimes M_n$ is positive.
\end{definition}

A linear map $S\in\mathrm{Lin}(M_n,M_n)$ is called \emph{completely positive}, if it is $k$-positive for all $k\in \mathbb{N}$. It is known that a linear map is completely positive if and only if it is $n$-positive and in this case it admits a Kraus representation
\[
S(X) = \sum_{i=1}^p L_i^* X L_i, \qquad X\in M_n,
\]
with certain matrices $L_1,\ldots,L_p\in M_n$.

\begin{definition}
A map $S\in\mathrm{Lin}(M_n,M_n)$ is called \emph{$k$-superpositive}, if it admits a Kraus representation $S(X) = \sum_{i=1}^p L_i^* X L_i$ where all the matrices $L_1,\ldots,L_p$ have rank less than or equal to $k$.
\end{definition}

We will denotes the cones of $k$-positive and $k$-superpositive maps by $\mathcal{P}_k$ and $\mathcal{S}_k$, resp. We have $\mathcal{PM}=\mathcal{P}_1$. Denoting by $\mathcal{CP}$ the cone of completely positive maps, we have furthermore $\mathcal{P}_n=\mathcal{CP}=\mathcal{S}_n$.

\begin{definition}
A map $S\in \mathrm{Lin}(M_n, M_n)$ is called entanglement breaking if for any $X\in (M_k\otimes M_n)^+$ and $k\ge 1$, $(\mathrm{Id}_k \otimes S)X $ is separable (i.e., in the closed convex hull of $M_k^+\otimes M_n^+$). 
\end{definition}
We denote the set of all entanglement breaking maps on $M_n$ by $\mathcal{EB}$.
It can be shown that the cone of $1$-superpositive maps coincides with the set of entanglement breaking maps.

\begin{proposition}
$\mathcal{S}_1= \mathcal{EB}$.
\end{proposition}

\begin{proof}
\cite{hsr}.
\end{proof}

In the terminology of \cite{cmw} we can define $k$-entanglement breaking maps.
\begin{definition}
A linear map $S \in \mathrm{Lin}(M_n, M_n)$ is called $k$-entanglement breaking if it is $k$-positive and $(\mathrm{Id}_k \otimes S)X$ is separable whenever $X\in (M_k\otimes M_n)^+$.
\end{definition}
We denote the cone of $k$-entanglement breaking maps by $\mathcal{EB}_k$. We have $\mathcal{EB}=\mathcal{EB}_n$.

These cones satisfy following chains of inclusions:
\begin{itemize}
\item [(i)]
$\mathcal{EB}=\mathcal{S}_1 \subset \mathcal{S}_k\subset \mathcal{S}_n=\mathcal{CP}=\mathcal{P}_n\subset \mathcal{P}_k\subset \mathcal{P}_1=\mathcal{PM}$,
\item[(ii)]
$\mathcal{EB}=\mathcal{EB}_n\subset \mathcal{EB}_k \subset \mathcal{P}_k \subset \mathcal{P}_1=\mathcal{PM}$,
\item[(iii)]
$\mathcal{EB}=\mathcal{S}_1 \subset \mathcal{S}_k \subset \mathcal{EB}_k^{\circ} \subset \mathcal{P}_1=\mathcal{PM}$.
\end{itemize}
It is easy to see that $\mathcal{P}_k$ is a closed cone and the cone $\mathcal{EB}_k$ is also closed, see \cite[Theorem 3.12]{dms}. Furthermore we have the duality relations that $\mathcal{P}_k^{\circ}=\mathcal{S}_k$ for $k=1,\ldots , n$, cf.\ \cite{ssz}.
In summary,
\[
\mathcal{P}_k^\circ = \mathcal{S}_k, \quad \mathcal{S}_k^{\circ}= \mathcal{P}_k \quad \text{and} \quad (\mathcal{EB}_k^{\circ})^{\circ}=\mathcal{EB}_k
\]
for k=1, \ldots, n.

In Subsection \ref{subsec-cones}, we will show that all these cones are solid and closed under composition.

\subsection{Non-unital semigroups}

We will be interested in semigroups of linear operators $(T_t)_{t\ge 0}$ (on some Banach space $V$), which do not start from the identity, i.e. we have $T_sT_t=T_{s+t}$ for all $s,t\ge 0$, but not necessarily $T_0=\mathrm{Id}$.
We still want $t\mapsto T_t$ to be continuous.

The semigroup property implies
\[
T_0^2=T_0
\]
i.e., $T_0$ is idempotent. Then we can decompose $V$ as $V=\mathrm{Im}(T_0)\oplus\mathrm{Ker}(T_0)$, where $\mathrm{Im}(T_0)=T_0(V)={\rm Ker}(\mathrm{Id}_V-T_0)$, and $\mathrm{Ker}(T_0)=\mathrm{Im}(\mathrm{Id}_V-T_0)=(\mathrm{Id}_V-T_0)(V)$ We assume that $T_0$ is bounded, so both subspaces are closed. With respect to this decomposition $T_0$ has the form
\[
T_0 =
\left(\begin{array}{cc}
\mathrm{Id}_{\mathrm{Im}(T_0)} & 0 \\
0 & 0
\end{array}\right)
\]

Furthermore, the semigroup property implies $T_0T_t=T_tT_0 = T_t$ for all $t\ge 0$. Therefore
\[
\mathrm{Ker}(T_0)\subseteq \mathrm{Ker}(T_t) \quad\mbox{ and }\quad \mathrm{Im}(T_t)\subseteq\mathrm{Im}(T_0).
\]
W.r.t.\ the decomposition $V=\mathrm{Im}(T_0)\oplus\mathrm{Ker}(T_0)$ we can write the $T_t$ as
\[
T_t =
\left(\begin{array}{cc}
\widetilde{T}_t & 0 \\
0 & 0
\end{array}\right)
\]
with some linear operators $\widetilde{T}_t\in B(\mathrm{Im}(T_0))$, which form a continuous semigroup $(\widetilde{T}_t)_{t\ge 0}$ with initial value $\widetilde{T}_0=\mathrm{Id}_{\mathrm{Im}(T_0)}$. This allows to extend classical results on unital semigroups to the non-unital case.

In our examples, if $T_0$ is a conditional expectation onto some unital *-subalgebra,  then we are lead to study semigroups $(\tilde{T}_t)_{t\ge 0}$ that preserve the corresponding cones of $T_0(M_n)\subseteq M_n$.

But in general $T_0$ need not be a conditional expectation, as the example in the following subsection shows.

\subsection{Examples of $k$-superpositive semigroups}

Since the identity map $\mathrm{Id}:M_n\to M_n$ is not $k$-superpositive for $k<n$, there exists no $k$-superpositive semigroup $(T_t:M_n\to M_n)_{t\ge0}$ with $T_0={\rm id}$. But there do exist semigroups of $k$-superpositive linear maps on $M_n$ that start with an idempotent $k$-superpositive map $T_0$. Very simple examples are given by $T_0(X) = PXP$ with $P$ a $k$-dimensional orthogonal projection.

Another class of examples are semigroups $T_t:M_2\otimes M_n\to M_2\otimes M_n$ of the form
\[
T_t\left(\left(\begin{array}{cc}
A & B \\
C & D
\end{array}\right)\right)
=
\left(\begin{array}{cc}
S_t(A) & 0 \\
0 & \alpha\big(S_t(A)\big)
\end{array}\right)
\]
with $(S_t)_{t\ge 0}$ a $k$-superpositive semigroup acting on $M_n$ and $\alpha:M_n\to M_n$ any $k$-superpositive linear map. If $\alpha$ is not a *-homomorphism on $\mathrm{Im}(S_0)$, then $T_0$ is a $k$-superpositive idempotent whose image is not a *-subalgebra, and so $T_0$ is not a conditional expectation.

\section{A Schoenberg correspondence for general non-unital semigroups}

We give a non-unital version of \cite[Lemma 2.1]{sch}.

For $X$ a Banach space, we will denote by $B_{X'}$ and $S_{X'}$ the unit ball and the unit sphere of the dual space $X'$. If $C\subseteq X$ is a cone in $X$ then $C^{\circ}$ denotes the dual cone in $X'$, i.e.
\[
C^{\circ} = \{\varphi\in X'; \forall v\in C, \varphi(v)\ge 0\}.
\]
Note that we have (for $C$ a closed convex cone)
\[
(C^{\circ})^{\circ} = C.
\]

\begin{theorem}\label{thm-main}
Let $A$ be a real Banach algebra with a closed convex cone $C\subseteq A$ with non-empty interior. Let $a_0\in C$ be an idempotent such that for any $c\in C$, we have $a_0ca_0\in C$.

Furthermore, it is assumed that for any $c\in C$ we have $c^n\in C$ for $n\ge 1$.

Then, for any $b\in A$ such that $ba_0=a_0b=b$, the following statements are equivalent.

\begin{description}
\item[(i)]
$b$ is $a_0$-conditionally positive on $C^{\circ}$, i.e., $\varphi(b)\ge 0$ for all $\varphi\in C^{\circ}$ with $\varphi(a_0)=0$.
\item[(ii)]
$\exp_{a_0}(tb):=\lim_{n\to\infty}\left(a_0+\frac{tb}{n}\right)^n\in C$ for all $t\ge 0$.
\end{description} 
\end{theorem}

\begin{remark}
One might wonder, if the condition that $a_0ca_0\in C$ for all $c\in C$ could be weakened or omitted. We use this condition in step IV of the proof below, and we have not been able to remove it. But we do not know of any counter-exemple that would show that it can not be omitted, either.
\end{remark}

\begin{proof}

\noindent
\textbf{(ii)}$\Rightarrow$\textbf{(i)}: This follows by diffentiation at $t=0$. If $\varphi\in C^{\circ}$ is such that $\varphi(a_0)=0$, then
\[
\varphi(b) = \lim_{\genfrac{}{}{0pt}{2}{t\to0}{t>0}}\varphi\left(\frac{\exp_{a_0}(tb)-\exp_{a_0}(0)}{t}\right) = 
\lim_{\genfrac{}{}{0pt}{2}{t\to0}{t>0}}\frac{1}{t}\varphi\big(\exp_{a_0}(tb)\big)\ge 0,
\]

\noindent
\textbf{(i)}$\Rightarrow$\textbf{(ii)}:
We have
\begin{equation}
\exp_{a_0}(tb) := \lim_{n\to\infty} \left(a_0+\frac{tb}{n}\right)^n = a_0 + \sum_{n=1}^\infty \frac{(tb)^n}{n!}
\end{equation}
and we want to show that this quantity is positive for $t\ge0$ if $b$ satisfies condition (i). Without loss of generality we can take $t=1$.

We will prove in four steps that $\exp_{a_0}(b)$ is positive.

\noindent
\textbf{Step I:} For any interior point $c\in C$ there exists $\delta>0$ such that
\[
\forall \varphi\in C^{\circ}, \quad \varphi(c)\ge \delta\|\varphi\|.
\]
Indeed, let $c\in C$ be an interior point. Then there exists $\delta>0$ such that $c+\delta B_{A}\subseteq C$, where $B_A$ is the unit ball in A.
Therefore for $v\in B_{A}$,
\[
\varphi(c\pm\delta v) \ge 0,
\]
and
\[
\varphi(c)\ge \delta \sup_{v\in B_{A}} |\varphi(v)| = \delta\|\varphi\|.
\]

\noindent
\textbf{Step II:} for any $\rho>0$ there exists $\eta>0$ such that
\[
\forall \varphi\in C^{\circ}\cap B_{A'}, \quad \varphi(a_0) < \eta\Rightarrow \varphi(b)>-\rho. 
\]
Indeed, fix $\rho>0$ and set
\[
V_n(\rho) = \left\{\varphi\in C^{\circ}\cap B_{A'}; \varphi(a_0) \le \frac{1}{n} \text{ and }\varphi(b) \le - \rho\right\}
\]
By the $a_0$-conditional positivity of $b$ w.r.t.\ $C^{\circ}$, we have
\[
\bigcap_{n\ge 1} V_n(\rho) = \emptyset.
\] 
Since $C^{\circ}\cap B_{A'}$ is compact by the Banach-Alaoglu theorem in the weak-* topology, and since the $V_n(\rho)\subseteq C^{\circ}\cap B_{A'}$ are weak-* closed, there exists $n_0\in\mathbb{N}$ such that $V_{n_0}(\rho)=\emptyset$. Take $\eta=\frac{1}{n_0}$.

\noindent
\textbf{Step III:}
For any $\varepsilon>0$ and $c\in C$ an interior point,
there exists $n_0\in\mathbb{N}$ such that for all $n\ge n_0$,
\[
\forall \varphi\in C^{\circ}\cap S_{A'},\quad \varphi\left(a_0+\frac{b+\varepsilon c}{n}\right) \ge 0,
\]
where $S_{A'}$ denotes the unit sphere in $A'$. 
Let $\delta>0$ be the real number guaranteed by Step I such that $\varphi(c)\ge \delta \|\varphi\|$ for all $\varphi\in C^{\circ}$. Let $\eta>0$ be the real number guaranteed by Step II such that for all $\varphi\in C^{\circ}\cap B_{A'}$ with $\varphi(a_0) < \eta$ we have $\varphi(b)\ge-\varepsilon\delta$. 

Let $\varphi\in C^{\circ}\cap S_{A'}$. We distinguish two cases, according to the value of $\varphi$ on $a_0$.

\noindent
\textbf{Case $\varphi(a_0)<\eta$:} in this case we have
\[
\varphi\left(a_0+\frac{b+\varepsilon c}{n}\right) = \underbrace{\varphi(a_0)}_{\ge 0} + \frac{1}{n}\big(\underbrace{\varphi(b)}_{\ge -\varepsilon\delta}+\varepsilon\underbrace{\varphi(c)}_{\ge \delta}\big) \ge 0.
\]

\noindent
\textbf{Case $\varphi(a_0)\ge\eta$:} now we get
\[
\varphi\left(a_0+\frac{b+\varepsilon c}{n}\right) = \underbrace{\varphi(a_0)}_{\ge \eta} + \frac{1}{n}\varphi(b+\varepsilon c) \ge \eta-\frac{\|b+\varepsilon c\|}{n},
\]
which is positive as soon as $n \ge \frac{\|b+\varepsilon c\|}{\eta}$.

\noindent
\textbf{Step IV:}
By the Bipolar theorem, this means that for any $\varepsilon>0$ and $c\in C$ an interior point there exists an $n_0\in\mathbb{N}$ such that for all $n\ge n_0$, $a_0+\frac{b+\varepsilon c}{n}\in C$.
Then we have $a_0\left(a_0+\frac{b+\varepsilon c}{n}\right)a_0=a_0+\frac{b+\varepsilon a_0ca_0}{n}\in C$.

Since $C$ is stable under taking powers and closed, we get
\[
\exp_{a_0}(b+\varepsilon a_0ca_0)=\lim_{n\to\infty}\left(a_0+\frac{b+\varepsilon a_0ca_0}{n}\right)^n\in C.
\]
To conclude the proof we let $\varepsilon\searrow0$.
\end{proof}

\section{Semigroups of $k$-(super)positive, $k$-entanglement breaking linear maps on $M_n$}
\label{sec:sup-1}

Now we apply Theorem \ref{thm-main} to the algebra $\mathrm{Lin}(M_n,M_n)^{\mathrm{her}}$ of hermitianity preserving linear maps from $M_n$ to $M_n$.

\subsection{Basic properties of several cones relevant to quantum information}
\label{subsec-cones}

First we verify some basic properties of the cones of $k$-(super)positive and $k$-entanglement breaking linear maps that will insure that we can apply Theorem \ref{thm-main}.

\begin{proposition}\label{cones}
The cones $\mathcal{P}_k,\mathcal{S}_k, \mathcal{EB}_k, \mathcal{EB}_k^{\circ}\subseteq \mathrm{Lin}(M_n,M_n)^{\mathrm{her}}$, $k=1,\ldots,n$, are closed, pointed and solid. Furthermore they are stable under composition.
\end{proposition}
\begin{proof}
Stability under composition is easy to check for $k$-positive maps and therefore for CP maps.
For $k$-superpositive maps, if $S,T\in \mathcal{S}_k$ have Kraus representations
\[
S(X)=\sum_{i=1}^p L_i^*XL_i \quad\text{ and }\quad T(X)=\sum_{j=1}^q K_j^* X K_j,
\]
with Kraus operators of rank less than or equal to $k$, then so does their composition $S\circ T$,
\[
S\circ T(X) = \sum_{i=1}^p \sum_{j=1}^q (K_jL_i)^* X K_j L_i,
\]
since $\mathrm{rank}(K_jL_i)\le \min\{{\mathrm{rank}(L_i),\mathrm{rank}(K_j)}\}\le k$.
For the stability of composition of $k$-entanglement breaking maps, see \cite[Theorem 5.4]{dms}. For the dual cone $\mathcal{EB}_k^\circ$, it follows from the characterisation given in \cite[Equation (3.5), Theorem 3.11]{dms}. Indeed, suppose that $S_i$ are limits of convex combinations of the form given in \cite{dms}, i.e., $S_i=\lim \sum \lambda_p^{(i)} \Gamma_{p}^{(i)}\circ \Psi_p^{(i)}$, with $\Gamma^{(i)}_p:M_k\to M_n$ positive and $\Psi_p^{(i)}:M_n\to M_k$ completely positive for $i=1,2$ (where we suppressed the index for the limit). Then $\Gamma_p^{(1)}\circ \Psi_p^{(1)}\circ \Gamma_p^{(2)}:M_k\to M_n$ is again positive, and so
\[
S_1\circ S_2 = \lim \sum_{p,p} \lambda^{(1)}_p \lambda^{(2)}_{p'} \left( \Gamma_p^{(1)}\circ \Psi_p^{(1)}\circ \Gamma_{p'}^{(2)}\right) \circ \Psi^{(2)}_{p'}
\]
is again of the form given in \cite[Equation (3.5), Theorem 3.11]{dms}.

To show that all these cones are solid (i.e., $C+(-C)=\mathrm{Lin}(M_n,M_n)$, or, equivalently, they have non-empty interior) and pointed (i.e., $C\cap(-C)=\{0\}$), we will show below in Lemma \ref{lem-D} that the depolarising channel $P:M_n\to M_n$, $P(X)=\mathrm{Tr}(X)\frac{1}{n}I_n$ is in the interior of $\mathcal{S}_1=\mathcal{EB}$. Since all other cones contain $\mathcal{EB}$, they are also solid. Which implies that they are pointed, as duals of solid cones.
\end{proof}

\begin{lemma}\label{lem-D}
The depolarising channel $P:M_n\to M_n$, $P(X)=\mathrm{Tr}(X)\frac{1}{n}I_n$ belongs to the interior of the cone $\mathcal{S}_1=\mathcal{EB}$ of entanglement breahing maps.
\end{lemma}
\begin{proof}
We will use the duality between $\mathcal{EB}=\mathcal{S}_1$ and $\mathcal{PM}=\mathcal{P}_1$. To show that $P$ belongs to the interior of $\mathcal{EB}$, it is sufficient to show that
\[
\forall T\in \mathcal{PM}=\mathcal{EB}^\circ, T\not=0 \quad\Rightarrow\quad \langle T,P\rangle >0.
\]
Let $T\in \mathcal{PM}$ such that $\langle T,P\rangle =0$. 
Choose an orthonormal basis $(u_i)_{i=1,\ldots,n}$ of $\mathbb{C}^n$. 
Then the set of rank $1$ operators $\big(|u_i\rangle\langle u_j|\big)_{i,j=1,\ldots,n}$ is an orthonormal basis of $M_n$ and we have
\[
0 = \langle T,P\rangle = \sum_{i,j=1}^n \Big\langle T\big(|u_i\rangle\langle u_j|\big), \underbrace{P\big(|u_i\rangle\langle u_j|\big)}_{=\frac{\delta_{ij}}{n}I_n}\Big\rangle 
= \frac{1}{n} \sum_{i=1}^n \mathrm{Tr}\Big(T\big(|u_i\rangle\langle u_i|\big)^* \Big)
\]
Since $T$ is a positive map, each term in this sum is positive, and therefore has to vanish. Since the orthonormal basis $(u_i)_{i=1,\ldots,n}$ is arbitrary, this implies
\[
\forall u\in \mathbb{C}^n,\qquad \mathrm{Tr}\Big(T\big(|u\rangle\langle u|\big)\Big) =0.
\]
and therefore $T=0$.
\end{proof}

\subsection{Application of the Schoenberg correspondence}

We now apply the Schoenberg correspondence to maps on $M_n$ and the cones introduced earlier.

\begin{theorem}\label{cone k-pos}
Let
\[
A=\mathrm{Lin}(M_n,M_n)^{\mathrm{her}} = \big\{ T\in \mathrm{Lin}(M_n,M_n); T\circ * = *\circ T\big\}
\]
and let $C\subseteq A$ be one of the cones $\mathcal{PM}=\mathcal{P}_1$, $\mathcal{P}_2,\ldots,\mathcal{P}_{n-1}$, $\mathcal{P}_n=\mathcal{CP}=\mathcal{S}_n$, $\mathcal{S}_{n-1},\ldots,\mathcal{S}_2$, $\mathcal{S}_1=\mathcal{EB}=\mathcal{EB}_n$, $\mathcal{EB}_{n-1},\ldots \mathcal{EB}_1$, $\mathcal{EB}_1^\circ,\ldots,\mathcal{EB}_n^\circ=\mathcal{S}_1^\circ=\mathcal{PM}$.

Fix an idempotent map $T_0\in C$. Then for $S\in A$ with $S\circ T_0=T_0\circ S=S$ the following are equivalent.
\begin{itemize}
\item[(i)]
We have $\exp_{T_0}(tS)=T_0+\sum_{n=1}^\infty \frac{t^n S^{\circ n}}{n!}\in C$ for all $t\ge 0$;
\item[(ii)]
$S$ is $T_0$-conditionally positive on $C^\circ$, i.e., we have
\[
\forall v\in C^\circ,\quad \langle v, T_0\rangle =0 \quad\Rightarrow\quad \langle v, S\rangle \ge 0.
\]
\end{itemize}
\end{theorem}
\begin{proof}
The hermitianity preserving maps on $M_n$ form a real Banach algebra, when we equip it with the norm induced by the operator norm on $M_n\cong \mathrm{Lin}(\mathbb{C}^n,\mathbb{C}^n)$.

Proposition \ref{cones} ensures that all the cones are convex, solid, and pointed, so they have in particular a non-empty interior. Furthermore, they are closed under composition, so we have $S^{\circ n}\in C$ and $T_0\circ S\circ T_0\in C$ for any $S,T_0\in C$ and $n\ge 1$.

Therefore we can apply Theorem \ref{thm-main} to any pair $(T_0,S)$, with $T_0$ an idempotent in $C$ and $S\in A$ such that $S\circ T_0=T_0\circ S=S$, and the result follows.
\end{proof}

\begin{corollary}\label{ex-cone}
Take $C=\mathcal{P}_k$ with $k\in\{1,\ldots,n\}$ and $T_0=\mathrm{Id} = T_{I\otimes I}$.

A semigroup $T_t=\exp(tS)$ with generator $S=T_W$ where $W=\sum A_i\otimes B_i\in (M_n\otimes M_n^{op})^{\mathrm{sa}}$, consists of $k$-positive maps for all $t\ge 0$ if and only if
\[
\forall V\in M_n, \quad \big(\mathrm{rank}(V)\le k \text{ and }\mathrm{Tr}(V)=0\big) \quad \Rightarrow \quad \sum \mathrm{Tr}(A_i V^*)\mathrm{Tr}(B_i V) \ge 0.
\]
\end{corollary}

\begin{proof}
Here we consider unital semigroups, i.e., $T_0=\mathrm{Id}$, which belongs to $\mathcal{P}_k$ for all $k\ge 1$. Furthermore, the condition $S\circ T_0=T_0\circ S=S$ now holds for any $S\in \mathrm{Lin}(M_n,M_n)^{\mathrm{her}}$.

We know that
\[
C^\circ = \mathcal{P}_k^\circ = \mathcal{S}_k = \text{convex hull of }\{ T_{V\otimes V^*}; V\in M_n, \mathrm{rank}(V)\le k\},
\]
cf.\ \cite{ssz}.

Note that
\[
\langle T_{V\otimes V^*},\mathrm{Id}\rangle = \mathrm{Tr}(V^*\otimes V) = \left|\mathrm{Tr}(V)\right|^2,
\]
since, by Proposition \ref{isomorphism}, $T$ is an isomorphism of Hilbert spaces.

By Theorem \ref{cone k-pos}, $\exp(tS)\in \mathcal{P}_k$ for all $t\ge 0$, iff $S$ is $\mathrm{Id}$-conditionally positive on $\mathcal{S}_k=\mathcal{P}_k^\circ$, i.e. , if
\[
\forall \varphi\in \mathcal{S}_k, \quad \langle\varphi,\mathrm{Id}\rangle= 0 \quad\Rightarrow\quad \langle \varphi, S\rangle \ge 0.
\]
Let us check that it is sufficient to verify this for $\varphi\in\{ T_{V\otimes V^*}; V\in M_n, \mathrm{rank}(V)\le k\}$. Indeed, this set generates $\mathcal{S}_k$. And, since for a convex combination $\varphi= \sum \lambda_i T_{V_i\otimes V_i^*}$ with $\lambda_i>0$, $\sum \lambda_i=1$, we have
\[
\left\langle \sum \lambda_i T_{V_i\otimes V_i^*}, \mathrm{Id}\right\rangle =  \sum \lambda_i \mathrm{Tr}(V_i^*\otimes V_i) = \sum \lambda_i \left|\mathrm{Tr}(V_i)\right|^2,
\]
we see that the condition $\langle \varphi,\mathrm{Id}\rangle=0$ is satisfied for a convex combination iff it is satisfied for each term.

If $S=T_W$ with $W=\sum A_i\otimes B_i\in (M_n\otimes M_n^{op})^{\mathrm{sa}}$, then
\[
\langle T_{V \otimes V^*}, T_W \rangle = \langle V\otimes V^*, W\rangle = \mathrm{Tr}\left(V^* A_i\otimes V B_i\right) = \mathrm{Tr}(A_i V^*)\mathrm{Tr}(B_i V),
\]
which completes the proof.
\end{proof}
Let us now consider the special case of unital CP semigroups.

Let $\{B_j\}_{j=1,\ldots, n^2}$ be an orthonormal basis of $M_n$ such that $B_1= I_n$ and $\mathrm{Tr}(B_j)=0$ for $2\le j \le n^2$.
Taking $\{B_i\otimes B_j\}_{p,q=1}^{n^2}$ as a basis of $M_n\otimes M_n^{op}$, or equivalently $\{B_i\otimes B_j^*\}_{i,j=1}^{n^2}$ of $M_n\otimes M_n^*$ further we write $S= \sum_{i,j=1}^{n^2} D_{ij}B_i\otimes B_j^*$, where $D_{ij}$ is a $n^2 \times n^2$ matrix in $\mathbb{C}$, so that the map $T_S$ is given by
\begin{equation}\label{generator}
\Psi (X)= T_S(X)= \sum_{i,j=1}^{n^2} D_{ij}B_i X B_j^*
\end{equation}

for all $X\in M_n$.
The semigroup $\exp{(t\Psi)}$ is Hermitianity preserving if and only if the generator $T_S$ is hermitianity perserving if and only if the matrix $(D_{ij})$ is hermitian. 

\begin{proposition}\label{CP semigroup}
A semigroup $T_t= \exp{t \Psi}$ is completely positive if and only if the matrix $(D_{ij})_{i,j=1}^{n^2}$ is hermitian and for any $v=(0,v_2,v_3,\cdots, v_{n^2})\in \mathbb{C}^{n^2}$ we have
\[
\langle v|(D_{ij})_{i,j=1}^{n^2}v\rangle \ge 0,
\]
\end{proposition}
\begin{proof}
Let $V$ be a $n\times n$ matrix with the basis decomposition $V=\sum_{j=1}^{n^2}{v_j}B_j$.
The trace conditions on $B_j$'s implies that $\mathrm{Tr}(V)=0$ if and only if $v_1=0$.
By a direct application of Corollary \ref{ex-cone} we see that the map $T_S$ with $S=\sum_{i,j=1}^{n^2}B_i\otimes B_j^*$ generates a completely positive semigroup (i.e. $n$-positive) if and only if 
for any $V\in M_n$ with $\mathrm{Tr}(V)=0$, we have $\sum_{i,j=1}^{n^2}D_{ij}\mathrm{Tr}(B_iV^*)\mathrm{Tr}(B_j^*V)\ge 0$.
Writing the basis decomposition of $V$ in this condition, we obtain 
\[
(\forall (0,v_2,v_3,\cdots, v_{n^2})\in \mathbb{C}^{n^2}) \implies \sum_{i,j=2}^{n^2}\sum_{k,l=2}^{n^2}D_{ij}\Bar{v_k}v_l \mathrm{Tr}(B_iB_k^*)\mathrm{Tr}(B_j^*B_l)\ge 0.
\]
Because of the orthonormality of the basis $B_i$'s, $\mathrm{Tr}(B_iB_k^*)=\delta_{ik}$ and $\mathrm{Tr}(B_j^*B_l)=\delta_{jl}$.
Thus the above condition becomes 
\[
\sum_{i,j=2}^{n^2}\sum_{k,l=2}^{n^2}D_{ij}\Bar{v_k}v_l \delta_{ik}\delta_{jl}= \sum_{i,j=2}^{n^2}D_{ij} \Bar{v_i}v_j \ge 0,
\]
which is the desired result.
\end{proof}

From this result, we can re-derive Lindblad\cite{lind}, Gorini, Kossakowski, and Sudarshan's theorem\cite{gks} on the generator of a CP semigroup.
\begin{theorem}
An identity preserving semigroup $\exp{(t\Psi)})$ is completely positive for all time $t\ge 0$ if and only if the generator $\Psi$ has the form
\[
\Psi(X)= i[H,X]+ \sum_{j=1}^k \{V_jXV^*_j-\frac{1}{2}(V_jV_j^*X + XV_jV_j^*)\}
\]
for all $X\in M_n$, where $H$ is an hermitian matrix and $V_1,\ldots,V_k\in M_n$.
\end{theorem}

\begin{proof}
We use the basis decomposition \eqref{generator} of the generator $\Psi$,
\[
\Psi(X)= D_{11}X + \sum_{i=2}^{n^2}D_{i1}B_iX + \sum_{j=2}^{n^2}D_{1j}XB_j^* + \sum_{i,j=2}^{n^2}D_{ij}B_iXB_j^*
\]
As $(D_{ij})$ is hermitian, if we denote $W:=\sum_{i=2}^{n^2}D_{i1}B_i$ and $\kappa:=D_{11}$ then $W^*= \sum_{j=2}^{n^2}D_{1j}B_j^*$ and $\kappa$ is a real number. 
From the previous Proposition \ref{CP semigroup} we know that $(D_{ij})_{i,j=2}^{n^2}$ is a positive matrix.
Therefore there exists $A_1,A_2,...,A_k\in M_{n^2-1}$ such that $(D_{ij})_{i,j=2}^{n^2}= A_1A_1^*+A_2A_2^*+...+A_kA_k^*$.
If we write $A_r=(a_r(p,q))_{1\le p,q\le n}$ then in terms of the coefficients of the matrix $A_r$ we have
\[
D_{ij}= \sum_{r=1}^k \sum_{p=2}^{n^2} a_r(i,p)\overline{a_r(j,p)}.
\]
Substituting these in the above expression of $\Psi$ and regrouping the terms we obtain
\begin{align*}
\Psi(X) &= \kappa X + WX + XW^* + \sum_{r=1}^k\sum_{p=2}^{n^2}\Big(\sum_{i=2}^{n^2}a_r(i,p)B_i\Big)X\Big(\sum_{j=2}^{n^2}a_r(j,p)B_j\Big)^* \\
& = \kappa X + WX + XW^* + \sum_{p=2}^{n^2}\sum_{r=1}^k V_{p,r}XV_{p,r}^*
\end{align*}
where $V_{p,r}:= \sum_{i=2}^{n^2}a_r(i,p)B_i$.
That the semigroup preserves the identity is equivalent to the generator mapping it to zero i.e. $\Psi(I_n)=0$. Plugging in this condition we get
\[
0= \kappa + W + W^* + \sum_{p}\sum_{r}V_{p,r}V_{p,r}^*.
\]
So we can set $W= iH - \frac{1}{2}\kappa -\frac{1}{2}\sum_{p}\sum_{r}V_{p,r}V_{p,r}^*$, where $H$ is a Hermitian matrix and substituting it in the above expression of $\Psi(X)$, we have

\begin{align*}
\Psi(X) &= \kappa X + \left(iH-\frac{1}{2}\kappa-\frac{1}{2}\sum_{p}\sum_{r}V_{p,r}V_{p,r}^*\right)X + X\left(-iH-\frac{1}{2}\kappa - \frac{1}{2}\sum_{p}\sum_{r}V_{p,r}V^*_{p,r}\right) \\ 
&+ \sum_{p}\sum_{r}V_{p,r}XV_{p,r}^* \\
&= i[H,X] + \sum_{p}\sum_{r}\Big\{ V_{p,r}XV_{p,r}^* -\frac{1}{2}\Big(V_{p,r}V_{p,r}^*X + XV_{p,r}V^*_{p,r}\Big) \Big\}.
\end{align*}
\end{proof}

\section{Positive Semigroups on $M_2$} 
In \cite{car} positive semigroups on $M_2$ have been characterized in terms of their generators.  Here we give another characterization, following the discussion above. 

The Pauli matrices are unitary matrices 
\[
\begin{array}{cccc}
\sigma_0= \frac{1}{2}
\begin{bmatrix}
1 & 0 \\
0 & 1
\end{bmatrix}, & 
\sigma_1 =\frac{1}{2}
\begin{bmatrix}
0 & 1 \\
1 & 0
\end{bmatrix}, &
\sigma_2= \frac{1}{2}
\begin{bmatrix}
0 & -i \\
i & 0
\end{bmatrix}, &
\sigma_3= \frac{1}{2}
\begin{bmatrix}
1 & 0 \\
0 & -1
\end{bmatrix},   
\end{array}
\]
which form an orthogonal basis of $M_2$ with respect to Hilbert-Schmidt inner product.
Moreover, they satisfy the following relations 
\[
\begin{array}{ccc}
\sigma_p^2= \sigma_0^2, & \sigma_p \sigma_q=- \sigma_q \sigma_p, &  \sigma_p\sigma_q= i\sigma_r\sigma_0
\end{array}
\]
if $(p,q,r)\in \{(1,2,3),(2,3,1),(3,1,2)\}$.
Decomposing a matrix $V$ with respect to the Pauli basis i.e $V = \sum_{p=0}^3 v_p\sigma_p$, we observe that $V$ has the form 
\[
\frac{1}{2}
\begin{bmatrix}
v_0+v_3 & v_1-iv_2\\
v_1+iv_2 & v_0-v_3
\end{bmatrix}.
\]
It is easy to see that 
\begin{itemize}
\item[(i)]
$V$ has trace zero if anf only if $v_0=0$,
\item[(ii)]
it has rank one if and only if $V\not=0$ and $\mathrm{det}(V)=0$, i.e., $v_0^2=v_1^2+v_2^2 +v_3^2$.
\end{itemize}
For a semigroup $\exp{(t\Phi)}$ on $M_2$ we can write the generator again in terms of Pauli basis as in \eqref{generator},
\[
\Phi= T_W, \quad \text{where} \quad W=\sum_{p,q=0}^3 D_{p,q}\sigma_p\otimes \sigma_q.
\]

\begin{proposition}
A semigroup $\exp{(t\Psi)}$, $t\ge 0$, is positive if and only if the matrix $(D_{pq})_{p,q=0}^3$ is hermitian and for all $v= (v_1,v_2,v_3)\in \mathbb{C}^3$ with $v_1^2+v_2^2+v_3^2=0$ we have 
\begin{equation}
\langle v| (D_{pq})_{p,q=1}^3 v\rangle\geq 0
\end{equation}
where $(D_{pq})_{p,q=1}^3$ is the $3\times 3$ submatrix of the matrix $(D_{ij})_{i,j=0}^3$, defined above.
Moreover, it is identity preserving if and only if the following relations are satisfied,
\[
\sum_{p=0}^3D_{pp}=0, \quad \text{and} \quad (D_{p0}+D_{0p})+i(D_{qr}-D_{rq})=0 
\]
for $(p,q,r)\in \{(1,2,3),(2,3,1),(3,1,2)\}$.
\end{proposition}

\begin{proof}
As a direct application of Corollary \ref{ex-cone}, it follows that the semigroup $\exp{(t\Psi)}$ is positive iff for any $V\in M_2$ with $\mathrm{rank}(V)=1$ and $\mathrm{Tr}(V)=0$
\begin{equation}\label{positive_criterion}
\sum_{p,q=0}^3D_{pq}\mathrm{Tr}(\sigma_pV^*)\mathrm{Tr}(\sigma_qV)\geq 0
\end{equation}
As we observed, expanding $V$ in the Pauli basis $V=\sum_{j=0}^3v_j\sigma_j$ the rank and trace conditions translate into $v_0=0$ and $\sum_{j=1}^3v_j^2=0$.   
Plugging this decomposition of $V$ into the expression \eqref{positive_criterion}, it becomes
\begin{align*}
\sum_{p,q=0}^3 \sum_{j,k=0}^3 D_{pq} \overline{v}_j v_k\mathrm{Tr}(\sigma_p\sigma_j)\mathrm{Tr}(\sigma_q\sigma_k)
&= \sum_{p,q=0}^3\sum_{j,k=0}^3 \overline{v}_j v_k\delta_{pj}\delta_{qk}\\
&= \frac{1}{4}\sum_{p,q=1}^3D_{pq}\overline{v}_p v_q \geq 0,
\end{align*}
which gives the desired inequality.

To obtain the conditions of hermitianity and identity preserving property, it is easy to see that
\begin{itemize}
\item [(i)] $\exp{(t\Psi)}$ is hermitianity preserving if and only if the generator $\Psi=T_W$ has the same property, which is equivalent to the matrix $(D_{pq})_{p,q=0}^3$ being hermitian.
\item[(ii)] The semigroup preserves identity if and only if the generator $\Psi$ takes identity to zero i.e. $T_S(\sigma_0)=0$.
\end{itemize}
\end{proof}

\section{A $4$-parameter family of semigroups}

\subsection{The examples}
We now study the linear combinations of some well known maps, namely the depolarizing channel, the transpose, the conditional expectation onto diagonal matrices, and the identity map, and discuss the criteria for positivity, k-positivity or complete positivity of such combinations.
Next we will use these examples to generate identity preserving semigroups, which are of our ineterest.
We have already defined the \textit{depolarising channel} in Lemma \ref{lem-D}, as the linear map $P:M_n\to M_n$ satisfying
\[
P(X) = \frac{1}{n}\mathrm{Tr}(X) I_n.
\]
We have
\[
P(X) = \frac{1}{n} \sum_{j,k=1}^n \langle e_j, X e_j\rangle |e_k\rangle\langle e_k| = \sum_{j,k=1}^n \frac{1}{\sqrt{n}} |e_k\rangle\langle e_j| X \left(\frac{1}{\sqrt{n}}|e_k\rangle\langle e_j|\right)^*,
\]
which shows that $P\in \mathcal{S}_1=\mathcal{EB}$.
The Choi-Jamio{\l}kowski matrix of $P$ is
\[
C_P = \sum_{j,k=1}^n E_{jk}\otimes P(E_{jk}) = \frac{1}{n} \sum_{j,k=1}^n E_{jj}\otimes E_{kk} = \frac{1}{n} I_n\otimes I_n.
\]

We consider also the \textit{transposition map}, $T(X) = X^T$. It is known that $T$ is positive, but not $2$-positive, i.e.
\[
T\in\mathcal{P}_1, \qquad T\not\in\mathcal{P}_2,
\]
for $n\ge 2$.
The Choi-Jamio{\l}kowski matrix of $T$ is
\[
C_T = \sum_{j,k=1}^n E_{jk}\otimes T(E_{jk}) =  \sum_{j,k=1}^n E_{jk}\otimes E_{kj}.
\]
\textit{Conditional expectation onto the the diagonal:}
Consider the linear map $D:M_n\to M_n$, $D(X) = (\delta_{jk} x_{jk})_{1\le j,k\le n}$ for $X=(x_{jk})_{1\le j,k\le n}\in M_n$. This map is the conditional expectation onto the *-subalgebra of diagonal matrices (w.r.t.\ the standard basis).

We have
\[
D(X) = \sum_{j=1}^n |e_j\rangle\langle e_j| X |e_j\rangle \langle e_j|,
\]
which shows that $D\in \mathcal{S}_1=\mathcal{EB}$.

We can furthermore show that $D$ belongs to the boundary of $\mathcal{CP}$, and therefore also to the boundary of $\mathcal{EB}$. Indeed, denote by $C\in M_n$ that matrix that cyclically permutes the vectors of the standard basis,
\[
C e_j = e_{j\oplus 1} = \left\{
\begin{array}{cc}
e_{j+1} & \text{ if } 1\le j\le n-1, \\
e_1 & \text{ if } j=n,
\end{array}\right.
\]
where we use $\oplus$ to denote the addition modulo $n$ in $\{1,\ldots,n\}$. 
Then the completely positive map $T_C$ with $T_C(X)= CXC^*$ acts as
\[
T_C(E_{jk}) = E_{j\oplus1 ,k\oplus 1}.
\]
Thus we have 
\[
\langle T_C,D\rangle = \sum_{j,k=1}^n \mathrm{Tr}\big( T_C(E_{jk})^* D(E_{jk})\big) = \sum_{j=1}^n \mathrm{Tr}( E_{j\oplus 1, j\oplus 1} E_{jj}) = 0.
\]

The Choi-Jamio{\l}kowski matrix of $D$ is
\[
C_D = \sum_{j,k=1}^n E_{jk}\otimes D(E_{jk}) =  \sum_{j=1}^n E_{jj}\otimes E_{jj}.
\]

We have
\[
C_\mathrm{Id} = \sum_{j,k=1}^n E_{jk}\otimes E_{jk}.
\]

We are interested in the $4$ parameter family of linear maps $\alpha P + \beta D + \gamma T+ \delta \mathrm{Id}$, $\alpha,\beta,\gamma,\delta\in\mathbb{R}$.
These are exactly the linear maps on $M_n(\mathbb{C})$ that are invariant under the action of the hyperoctahedral group as signed permutations and have been considered from that perspective in \cite{jppy}.
We wish to decide to which cone such a map belongs, depending on the values of the coefficients.
The Choi-Jamio{\l}kowski matrices $C_P,C_D, C_T, C_\mathrm{Id}\in M_n\otimes M_n$ commute, and therefore we can simultaneously diagonalise these four matrices.
The minimal polynomials of $C_D,C_T,C_\mathrm{Id}$ have degree two, and these matrices have two distinct eigenvalues. Computing the traces, we also get the multiplicities.
\[
\begin{array}{|c||c|c|}
\hline
\text{Eigenvalues of } C_D & \rho=0 & \rho = 1 \\[3pt] \hline
\mathrm{dim}\,\mathrm{ker} (C_D-\rho I_n\otimes I_n) & n^2-n & n \\[3pt] \hline \hline
\text{Eigenvalues of } C_T & \sigma = -1 & \sigma=1 \\[3pt] \hline
\mathrm{dim}\,\mathrm{ker} (C_T-\sigma I_n\otimes I_n) & \frac{1}{2}(n^2-n) & \frac{1}{2}(n^2+n) \\[3pt] \hline \hline
\text{Eigenvalues of } C_\mathrm{Id} & \tau = 0 & \tau=n \\[3pt] \hline
\mathrm{dim}\,\mathrm{ker} (C_\mathrm{Id}-\tau I_n\otimes I_n) & n^2-1 & 1 \\[3pt] \hline
\end{array} 
\]
In particular, $C_\mathrm{Id}$ is a multiple of the orthogonal projection onto $\Omega=\sum_{j=1}^n e_j\otimes e_j$, which is also an eigenvector for the other matrices.

Denote by
\[
V(\rho,\sigma,\tau) = \{ v\in \mathbb{C}^n\otimes \mathbb{C}^n: C_Dv=\rho v, C_T v = \sigma v, C_\mathrm{Id} v = \tau v\}
\]
the joint eigenspaces of $C_D,C_T,C_\mathrm{Id}$.
We can verify that, $V(\rho,\sigma,\tau)$ is non-null corresponding to the four triples $(\rho,\sigma, \tau)= (0,-1,0), (0,1,0), (1,1,0), (1,1,n)$.

\begin{proposition}\label{4para-maps}
Let $\alpha,\beta,\gamma,\delta\in\mathbb{R}$ and set
\[
\Phi(\alpha,\beta,\gamma,\delta)=\alpha P + \beta D + \gamma T + \delta \mathrm{Id}.
\]
Then $\Phi$ is completely positive if and only if $\alpha,\beta,\gamma,\delta$ satisfy the inequality
\[
    \alpha \ge \max\{ n|\gamma|, - n(\beta+\gamma), -n(\beta+\gamma)-n^2\delta \}.
\]
\end{proposition}
\begin{proof}
We check that the Choi-Jamio{\l}kowski matrix $C_{\Phi(\alpha,\beta,\gamma,\delta)}=\alpha C_P + \beta C_D + \gamma C_T + \delta C_\mathrm{Id}$ of $C_{\Phi(\alpha,\beta,\gamma,\delta)}$ has eigenvalues
\[
\frac{\alpha}{n} - \gamma,\quad \frac{\alpha}{n} + \gamma, \quad \frac{\alpha}{n}+\beta+\gamma, \quad \frac{\alpha}{n}+\beta+\gamma+n\delta,
\]
corresponding to the four triples $(\rho,\sigma,\tau)=(0,-1,0),(0,1,0),(1,1,0),(1,1,n)$ with non-trivial eigenspaces.
\end{proof}

Restricting to equivariant maps as in the terminology of \cite{bardet+collins+sapra} (see also \cite[Theorem 2.2]{collins+osaka+sapra}), we can also characterize the $k$-positivity criterion for subfamilies of the above mentioned family.

\subsubsection{$(1,0)$-unitarily equivariant case - linear combinations of identity and depolarising channel}
The identity map and the depo\-larising channel are $(1,0)$-unitarily equivariant (in the terminology of \cite[Definition 1.1 (iii)]{bardet+collins+sapra}. For $\Phi_{\alpha,\delta}=\alpha P+ \delta{\rm Id}$, $\Phi_{\alpha,\delta}(X) = \delta X + \frac{\alpha}{n} \mathrm{Tr}(X) I_n$ we
\[
\Phi_{\alpha,\delta}(UXU^*) = U \Phi_{\alpha,\delta}(X)U^*,
\]
for all $U,X\in M_n$ with $U$ unitary.

Let $1\le k\le n$.
By \cite[Theorem 2.4]{bardet+collins+sapra}, $\Phi_{\alpha,\delta}$ is $k$-positive, if and only if
\[
C^{(k)}_{\Phi_{\alpha,\delta}} = \sum_{i,j=1}^k E_{ij}\otimes \Phi_{\alpha,\beta}(E_{ij}) = \sum_{i,j=1}^k E_{ij}\otimes = \frac{\alpha}{n} I^{(k)}_n\otimes I_n + \delta \sum_{i,j=1}^k E_{ij}\otimes E_{ij} \in M_n\otimes M_n  
\]
is positive, where $I^{(k)}_n=\sum_{i=1}^k E_{ii}\in M_n$.

Note that $C^{(k)}_{\Phi_{0,1}} =\sum_{i,j=1}^k E_{ij}\otimes E_{ij}$ commutes with $I_n^{(k)}\otimes I_n$, satisfies $\left(C^{(k)}_{\Phi_{0,1}}\right)^2= k C^{(k)}_{\Phi_{0,1}}$, and is a multiple of the orthogonal projection onto $\Omega_k=\sum_{i=1}^k e_i\otimes e_i$. One can show that the eigenvalues of  $C^{(k)}_{\Phi_{\alpha,\delta}}=\frac{\alpha}{n} I^{(k)}_n\otimes I_n + \delta C^{(k)}_{\Phi_{0,1}}$ are given by
\[
\mathrm{spec}(C^{(k)}_{\Phi_{0,1}}) = \left\{0,\frac{\alpha}{n},\frac{\alpha}{n}+k\delta\right\}.
\]

We summarize our results in the following lemma.
\begin{lemma}\label{lem-10eq}
(\cite[Theorem 2]{tom})
Let $\alpha,\delta\in\mathbb{R}$.
The linear map $\Phi_{\alpha,\delta}$ is $k$-positive iff the matrix $C^{(k)}_{\Phi_{\alpha,\delta}}$ is positive iff $\alpha,\delta$ satisfy the following two inequalities
\[
\alpha \ge 0 \quad\text{ and }\quad \delta \ge - \frac{\alpha}{kn}.
\]
\end{lemma}

\subsubsection{$(0,1)$-unitarily equivariant case - linear combinations of transposition and depolarising channel}
The transposition and the depolarising channel are $(0,1)$-unitarily equivariant (in the terminology of \cite[Definition 1.1 (iii)]{bardet+collins+sapra}. For $\Psi_{\alpha,\gamma}=\alpha P + \gamma T$, $\Psi_{\alpha,\gamma}(X) = \gamma X^T + \frac{\alpha}{n}\mathrm{Tr}(X) I_n$, we have
\[
\Psi_{\alpha,\gamma}(UXU^*) = \overline{U}\Psi_{\alpha,\gamma}(X)U^T,
\]
for all $U,X\in M_n$ with $U$ unitary.

By \cite[Theorem 2.4]{bardet+collins+sapra}, $\Psi_{\alpha,\gamma}$ is $k$-positive, if and only if
\[
C^{(k)}_{\Psi_{\alpha,\gamma}} = \sum_{i,j=1}^k E_{ij}\otimes \Psi_{\alpha,\gamma}(E_{ij}) = \frac{\alpha}{n} I^{(k)}_n\otimes I_n + \gamma \sum_{i,j=1}^k E_{ij}\otimes E_{ji} \in M_n\otimes M_n  
\]
is positive, where $I^{(k)}_n=\sum_{i=1}^k E_{ii}\in M_n$.

Note that $C^{(k)}_{\Psi_{0,1}}=\sum_{i,j=1}^k E_{ij}\otimes E_{ji}$ satisfies
\[
I^{(k)}_n \otimes I_n  C^{(k)}_{\Psi_{0,1}} =C^{(k)}_{\Psi_{0,1}} =C^{(k)}_{\Psi_{0,1}} I^{(k)}_n\otimes I_n, 
\]
and $(C^{(k)}_{\Psi_{0,1}})^2=\sum_{i,j=1}^k E_{ii}\otimes E_{jj}=I^{(k)}_n\otimes I^{(k)}_n$. which implies
\[
\mathrm{spec}(C^{(1)}_{\Psi_{0,1}})=\{0,1\}\quad \text{ and } \quad
\mathrm{spec}(C^{(k)}_{\Psi_{0,1}}) = \{-1,0,1\} \text{ for } k\ge 2.
\]
We see that the eigenvalues of $C^{(k)}_{\Psi_{\alpha,\gamma}}=\frac{\alpha}{n} I^{(k)}_n\otimes I_n + \gamma C^{(k)}_{\Psi_{0,1}}$ are
\[
\mathrm{spec}(C^{(1)}_{\Psi_{\alpha, \gamma}})=\left\{0,\frac{\alpha}{n}, \frac{\alpha}{n}+\gamma \right\} \quad \text{and} \quad
\mathrm{spec}(C^{(k)}_{\Psi_{\alpha,\gamma}})=\left\{0, \frac{\alpha}{n},\frac{\alpha}{n}+\gamma,\frac{\alpha}{n}-\gamma \right\}
\]
for $k\ge 2$.

Therefore we have the following lemma.

\begin{lemma}\label{lem-01eq}
(\cite[Theorem 3]{tom})
Let $\alpha,\gamma\in\mathbb{R}$.
The linear map $\Psi_{\alpha,\gamma}$ is 
\begin{enumerate}
\item[i.]
$1$-positive iff $\alpha\ge 0$ and $\alpha\ge -n\gamma$.
\item[ii.]
$k$-positive for $k\ge 2$, iff $\alpha \ge n|\gamma|$.
\end{enumerate}
\end{lemma}

Now we discuss as an example the semigroup generated by $4$ parameter family $L= \alpha P + \beta D + \gamma T + \delta \mathrm{Id}$ and time evolution for specific cases in $M_2$. 
As all the operators $P, D, T, \mathrm{Id}$ commute, the semigroup generated is given by
\begin{align*}
\exp{(tL)}=  &
e^{t(\beta +\gamma + \delta)}(e^{t\alpha}-1)P + e^{t(\gamma + \delta)}(e^{t\beta}-1)D + e^{t\delta}\frac{(e^{t\gamma}-e^{-t\gamma})}{2}T \\
& + e^{t\delta}\frac{(e^{t\gamma}+e^{-t\gamma})}{2}\mathrm{Id}
\end{align*}

\subsection{Semigroups generated by the depolarising channel}

We consider the semigroup $T_{\alpha, \delta}= \exp{t(\alpha P + \delta \mathrm{Id})}$.
This semigroup is again a linear combination of the operator $P$ and identity,
\[
T_{\alpha,\delta}(t) = \exp\big(t(\alpha P + \delta\mathrm{Id})\big) = \underbrace{e^{\delta t}(e^{\alpha t} -1)}_{=:\alpha(t)} P + \underbrace{e^{\delta t}}_{=:\delta(t)}\mathrm{Id},
\]
which, by Lemma \ref{lem-10eq}, is $k$-positive iff 
\[
\alpha(t) \ge 0 \quad \text{ and } \quad \delta(t) \ge -\frac{\alpha(t)}{kn}
\]
This condition is satisfied for all $t\ge 0$ iff $e^{\alpha t}\ge 1$ for all $t\ge 0$ iff $\alpha \ge 0$.

If we consider the identity preserving semigroup generated by the operator $P$ and the identity operator i.e. the semigroup $T_{\alpha, -\alpha}$.
we can see that it is a convex combination of $P$ and the identity opearator
\[
T_{\alpha, -\alpha}= \exp{t\alpha(P-\mathrm{Id})}= (1-e^{-t\alpha})P + e^{-t\alpha}\mathrm{Id}.
\]
which is completely positive for all $t\geq 0$ iff $\alpha \ge 0$.
What is more, in this case the semigroup converges to $P$ as $t \to \infty$, which we know by Lemma \ref{lem-D} to be an interior point of the cone $\mathcal{EB}$ or $1$-superpositive.
So if $\alpha \ge 0$, the semigroup $T_{\alpha,-\alpha}$ enters the cone $\mathcal{EB}$ in finite time $t_1> 0$.
In the special case $n=2$, the Choi-Jamio{\l}kowski matrix of the semigroup is given by
\[
C_{T_{\alpha,-\alpha}}=
\begin{bmatrix}
\frac{1}{2}(1+e^{-t\alpha}) & 0 & 0 & e^{-t\alpha} \\
0 & \frac{1}{2}(1-e^{-t\alpha}) & 0 & 0 \\
0 & 0 & \frac{1}{2}(1-e^{-t\alpha}) & 0 \\
e^{-t\alpha} & 0 & 0 & \frac{1}{2}(1+e^{-t\alpha})
\end{bmatrix},
\]

To decide when the semigroup enters $\mathcal{EB}$ we use Peres-Horodeci (or PPT) criterion.
If $T$ is the transpose map then 
\[
(\mathrm{Id}\otimes T)C_{\exp{(t L)}}=
\begin{bmatrix}
\frac{1}{2}(1+e^{-t\alpha}) & 0 & 0 & 0 \\
0 & \frac{1}{2}(1-e^{-t\alpha}) & e^{-t\alpha} & 0 \\
0 & e^{-t\alpha} & \frac{1}{2}(1-e^{-t\alpha}) & 0 \\
0 & 0 & 0 & \frac{1}{2}(1+e^{-t\alpha})
\end{bmatrix}.
\]
So $T_{\alpha, -\alpha}$ is in $\mathcal{EB}$ if and only if $C_{T_{\alpha, -\alpha}}$ is separable if and only if $(\mathrm{Id}\otimes T)C_{T_{\alpha, -\alpha}}$ is positive, by the PPT criterion, see \cite{per,hhh}.
We can easily check that the above matrix is positive iff the determinant
\[
\begin{vmatrix}
\frac{1}{2}(1-e^{-t\alpha}) & e^{-t\alpha} \\
e^{-t\alpha}    & \frac{1}{2}(1-e^{-t\alpha})
\end{vmatrix}
= -\frac{3}{4}e^{-2t\alpha}-2e^{-t\alpha}+1.
\]
is positive.
Replacing $x= e^{-t\alpha}$ we observer that the polynomial $-\frac{3}{4}x^2-2x+1$ has the positive root $x_0:=-\frac{4}{3}+2\frac{\sqrt{7}}{3}$ in the interval $[0,1]$.
The polynomial is positive on the interval $[0,x_0]$ and negative on $[x_0, 1]$.
Thus the semigroup enters the cone $\mathcal{S}_1$ at the time
\begin{equation*}
t_1 = - \frac{1}{\alpha}\ln{\left(-\frac{4}{3}+\frac{2\sqrt{7}}{3}\right)}.
\end{equation*}

\subsection{Semigroups generated by the depolarizing channel and transposition}

We consider the identity preserving semigroup generated by the transposition and the depolarizing channel i.e. by the generator $L=\alpha P + \gamma T - (\alpha + \gamma)\mathrm{Id}$, for $\alpha, \gamma \in \mathbb{R}$.
Since all the quantities $T, P, Id$ commutes we have 
\begin{align*}
\exp(tL) &= \exp (t\gamma(T-Id))\exp(t\alpha(P-Id)) \\
&= e^{-t\alpha}\frac{1+e^{-2t\gamma}}{2}Id + e^{-t\alpha}\frac{1-e^{-2t\gamma}}{2}T + (1-e^{-t\alpha})P,
\end{align*}
which is a convex combination of the $\mathrm{Id}, T$ and $P$.

If we take $\alpha, \gamma >0$, then in general it is a positive semigroup but not necessarily completely positive.

If the parameter $\alpha=0$ then the semigroup is just convex combination of $\mathrm{Id}$ and $T$, which converges to $\frac{1}{2}(\mathrm{Id}+T)$ as $t \to \infty$.

If $\alpha>0$, then the semigroup eventually becomes completely positive, even $1$-superpositive as it converges to $P$.

We compute the Choi-Jamio{\l}kowski matrix of of the semigroup $\exp{(t L)}$,
\begin{equation}
C_{\exp{(t L)}}= \rho_t C_{\mathrm{Id}} + \mu_t C_T + \nu_t C_P,
\end{equation}
where $\rho_t=e^{-t\alpha}\frac{1+e^{-2t\gamma}}{2}$, $\mu_t =e^{-t\alpha}\frac{1-e^{-2t\gamma}}{2}$ and $\nu_t=(1-e^{-t\alpha})$.
Using the same arguments as in Proposition \ref{4para-maps}, the Choi-Jamio{\l}kowski matrix is positive iff the eigenvalues $\frac{\nu_t}{n}-\mu_t$ and $\frac{\nu_t}{n}+\mu_t$ are positive.
Combining these two conditions, we conclude that the semigroup $\exp{(tL)}$ becomes CP at time $t$ iff
\[
\frac{\nu_{t}}{n}\ge |\mu_{t}| \quad \text{ i.e. } \quad 2(e^{\alpha t}-1) \ge n|1-e^{-2t \gamma}|.
\]
The above inequaility shows that even if $\gamma >> \alpha> 0$ it is possible that the semigroup not CP for certain time but ultimately becomes CP and then superpositive after finite time.  
If $\gamma=\alpha >0$ then substituting $x=e^{\alpha t}$ in the above inequality gives the following criterion,
\[
2x^3-(2+n)x^2+n \ge 0.
\]
The polynomial $p(x)= 2x^3-(2+n)x^2+n$ has two positive roots- 1 and $\frac{n+\sqrt{n^2+8n}}{4}$, and $p(x)\ge 0$ for $x\ge \frac{n+\sqrt{n^2+8n}}{4}$.
The root $1$ corresponds to the time $t=0$. 
So the semigroup becomes CP at time $t_1= \frac{1}{\alpha}\ln{\frac{n+\sqrt{n^2+8n}}{4}}$.

We can find the time when the semigroup becomes $1$-superpositive using again PPT criterion for the case $M_2(\mathbb{C})$.
\[
(\mathrm{Id}\otimes T)C_{\exp{t L}}=
\begin{bmatrix}
\rho_t + \mu_t + \frac{1}{2}\nu_t & 0 & 0 & \mu_t\\
0 & \frac{1}{2}\nu_t & \rho_t & 0\\
0 & \rho_t & \frac{1}{2}\nu_t & 0\\
\mu_t & 0 & 0 & \rho_t + \mu_t + \frac{1}{2}\nu_t
\end{bmatrix}.
\]
The above matrix is positive iff the determinant 
\[
\begin{vmatrix}
\frac{1}{2}\nu_t & \rho_t\\
\rho_t & \frac{1}{2}\nu_t
\end{vmatrix}
= \frac{1}{4}\nu_t^2-\rho_t^2
\]
is positive, which gives the condition 
\begin{equation}
2 \leq e^{t \alpha} - e^{-2 t \gamma}.
\end{equation}
If $\gamma=\alpha >0$ then substituting $x=e^{t\alpha}$ in the abive inequality gives that
\[
x^3-2x^2-1\ge 0.
\]
If $\xi$ is the positive root of the polynomial then we see that at the time $t_2=\frac{1}{\nu}\ln{\xi}$ the semigroup becomes $1$-superpositive.


\subsection*{Acknowledgements}
P.C.\ and U.F.\ were supported by the ANR Project No.\ ANR-19-CE40-0002. Bhat is suppported by the J C Bose Fellowship JBR/2021/000024 of SERB(India).

\subsection*{Author contribution} All authors wrote and reviewed the manuscript.

\subsection*{Funding}
The authors received no funding for the preparation of this manuscript besides the grants mentioned above.

\section*{Declarations}

\subsection*{Conflict of interest} The authors declare that they have no conflict of interest.


\end{document}